\documentclass[10pt,reqno]{amsart}

\usepackage{amsthm}
\usepackage{amsmath}
\usepackage{amssymb}
\usepackage{oldgerm}
\usepackage{graphicx}
\usepackage{mathrsfs}
\usepackage{enumerate}
\usepackage{MnSymbol}
\usepackage{tikz}
\usepackage{mathtools}
\usepackage{tikz-cd}

\theoremstyle{plain}
\newtheorem{thm}{Theorem}[section]
\newtheorem{lem}[thm]{Lemma}
\newtheorem{prop}[thm]{Proposition}
\newtheorem{cor}[thm]{Corollary}
\theoremstyle{definition}

\newcommand*{\sheafhom}{\mathscr{H}\kern -.5pt om}

\newcommand{\pic}{\textnormal{Pic}}

\let\ker\relax

\newcommand{\sing}{\textnormal{Sing}}
\newcommand{\ker}{\textnormal{Ker}}
\newcommand{\defeq}{\vcentcolon=}
\newcommand{\sym}{\textnormal{Sym}}

\tikzset{node distance=1.2cm,auto}
\usetikzlibrary{matrix}

\begin{document}
\title[\parbox{11.5cm}{\centering{Variety of singular quadrics containing a projective curve}}]{Variety of singular quadrics containing a projective curve}
\author{\.Irfan Kadik\"oyl\"u}
\address{Humboldt-Universit\"at zu Berlin, Institut f\"ur Mathematik, 10099
Berlin}
\email{irfankadikoylu@gmail.com}

\maketitle
\begin{abstract}
We study the variety of quadrics of rank at most $k$ in $\mathbb{P}^r$, containing a general projective curve of genus $g$ and degree $d$ and show that it has the expected dimension in the range $g-d+r\leq 1$. By considering the loci where this expectation is not true, we construct new divisor classes in $\overline{\mathcal{M}}_{g,n}$. We use one of these classes to show that $\overline{\mathcal{M}}_{15,9}$ is of general type.
\end{abstract}

\section{Introduction}
One of the celebrated problems of the theory of algebraic curves is the maximal rank conjecture, which predicts that the natural multiplication maps 
\begin{equation*}
\mu_k: \sym^k H^0\left(C,L\right)\to H^0\left(C,L^{\otimes k}\right)
\end{equation*}
are of maximal rank (i.e. either injective or surjective) for a general choice of $C$ and $L$ in the range where the Brill-Noether number is nonnegative. The original formulation of the conjecture is due to Harris \cite{Ha} and it amounts to showing that the dimension of the variety of hypersurfaces containing $C$ is the least possible. There are plenty of partial results confirming the conjecture in different special cases. We refer the reader to \cite{Ka} for a short account of these partial results. We note in particular that the quadratic case (i.e. $k=2$) of the conjecture, which is also the focus of the present work, is completely proven in the papers \cite{BF} and \cite{JP}.

In this paper we study the quadratic case of the problem from a refined perspective by taking also the ranks of the quadrics into account. Precisely, we let $C$ be a general curve of genus $g$ and $\ell$ be a general $g^r_d$ on it. We denote by $Q_k(C,\ell)$ the projective variety of quadrics of rank at most $k$ containing the image of $C$ under the map given by the linear series $|\ell|$. Since the codimension of the variety of rank $k$ quadrics in $|\mathcal{O}_{\mathbb{P}^r}(2)|$ equals $r-k+2\choose 2$, the expected dimension of $Q_k(C,\ell)$ is equal to
\begin{equation*}
q(g,r,d,k) \defeq {r+2\choose 2}-{r-k+2\choose 2}-2d+g-2.
\end{equation*}
In Theorem \ref{main thm} we confirm this expectation in the range $g-d+r\leq 1$:
\begin{thm}\label{main thm}
Let $C$ be a general curve of genus $g$ and $\ell$ be a general $g^r_d$ on $C$ where $g-d+r\leq 1$ and the Brill-Noether number $\rho(g,r,d)$ is nonnegative. Then the variety $Q_k(C,\ell)$ is of pure dimension $q(g,r,d,k)$. In particular, $Q_k(C,\ell)=\emptyset$ if $q(g,r,d,k) < 0$.
\end{thm}
Some parts of the rank 4 case of Theorem \ref{main thm} are already known in the literature. The dimension of the variety $Q_4(C,K_C)$ was computed by Andreotti and Mayer \cite{AM} using the correspondence between rank 4 quadrics and pencils on the curve. Using the same method Zamora \cite{Za} has computed the dimension of $Q_4(C,\ell)$ for linear systems $\ell$ of large degree.

There are valid reasons to expect that the failure loci of Theorem \ref{main thm} give rise to interesting cycles in moduli spaces. In a recent work, Farkas and Rim\'{a}nyi have studied loci of this form in different moduli spaces and obtained numerous interesting divisor classes \cite{FR}. In an analogous way, we use Theorem \ref{main thm} to construct new divisors on $\overline{\mathcal{M}}_{g,n}$:

We fix integers $g,n,k$ such that $4\leq k\leq g-n$ and 
\begin{equation*}
q(g,g-n-1,2g-2-n,k)=-1,
\end{equation*}
and define the locus
\begin{equation*}
\mathfrak{Quad}^k_{g,n}=\left\lbrace[C,p_1,\dots,p_n]\in\mathcal{M}_{g,n}\mid \exists q\neq 0\in I_2(C, K_C-\sum_{j=1}^n p_j) , \textnormal{ rk}(q)\leq k \right\rbrace,
\end{equation*}
where we denote by $I_2\left(C, K_C-\sum_{j=1}^n p_j\right)$ the kernel of the map
\begin{equation*}
\sym^2 H^0\left(K_C-\sum_{j=1}^n p_j\right) \to H^0\left(K_C^{\otimes 2}-2\sum_{j=1}^n p_j\right).
\end{equation*}
It follows from Theorem \ref{main thm} that this locus is a proper closed subset of $\mathcal{M}_{g,n}$. In Theorem \ref{class} we compute the class of its closure in $\overline{\mathcal{M}}_{g,n}$. To state the theorem, we recall the definition of the divisor classes that generate the Picard group of $\overline{\mathcal{M}}_{g,n}$. We denote the first Chern class of the Hodge bundle by $\lambda$ and $\psi_j$ stands for the first Chern class of the pullback of the relative dualizing sheaf via the section $\sigma_j: \overline{\mathcal{M}}_{g,n} \to \overline{\mathcal{M}}_{g,n+1}$ corresponding to the marked point labelled with $j$. The class of the irreducible singular curves with a non-separating node is denoted by $\delta_{irr}$ and $\delta_{i:S}$ is the class of the locus of curves whose general element is a reducible curve consisting of two components of genus $g - i$ and $i$, where the points labelled by $S\subseteq \{1,\dots ,n\}$ lie on the genus $i$ component.
\begin{thm}\label{class}
The class of the divisor $\overline{\mathfrak{Quad}}^k_{g,n}$ is given by the following formula:
\begin{equation*}
\overline{\mathfrak{Quad}}^k_{g,n} =\alpha^k_{g,n}\cdot\left(a\cdot\lambda+ c\cdot\sum_{j=1}^{n}\psi_j-b_{irr}\cdot \delta_{irr}-\sum_{i,s\geq 0}b_{i:s}\cdot\sum_{|S|=s}\delta_{i:S}\right),
\end{equation*}
where
\begin{equation*}
\alpha^k_{g,n}=\prod_{t=0}^{g-n-k-1}\frac{{g-n+t\choose g-n-k-t}}{{2t+1\choose t}},\quad a = \frac{7 g-9 n+6}{g-n},\quad c = \frac{g+n-6}{g-n},\quad b_{irr} = 1,
\end{equation*}
and all other coefficients are $\geq 1$. For $k=4$, we can further compute that
\begin{equation*}
b_{0:s} =\frac{s (gs-3s+n-3)}{g-n}.
\end{equation*}
\end{thm}
The number $\alpha^k_{g,n}$ is the degree of the variety of quadrics of rank at most $k$ inside the variety of all quadrics in $\mathbb{P}^{g-n-1}$ (see \cite{HT}). If we specialize to the case of smooth quadrics (i.e. $k= g-n$) then $\alpha^k_{g,n} = 1$ and we recover our formula for the divisor class in \cite{Ka}. 

We also use a different construction to obtain a divisor class in $\overline{\mathcal{M}}_{15,8}$. Since, as pointed out earlier, the quadratic case of the maximal rank conjecture holds, for a general element $[C,p_1,\dots ,p_8]\in\mathcal{M}_{15,8}$ one has that
\begin{equation*}
\dim I_2(C, K_C-\sum_{j=1}^8 p_j) = 2.
\end{equation*}
Adopting the terminology in \cite{FR}, we call this pencil \emph{degenerate} if its intersection with the variety of singular quadrics is non reduced. In Theorem \ref{Divisor on M15,8} we show that this condition singles out a divisor in $\mathcal{M}_{15,8}$ and compute the class of its closure:
\begin{thm}\label{Divisor on M15,8}
The locus of pointed curves defined as
\begin{equation*}
\mathfrak{D}_{15,8}\defeq \left\lbrace [C,p_1,\dots,p_{8}]\in \mathcal{M}_{15,8}\mid I_2(C, K_C-\sum_{j=1}^8 p_j) \textnormal{ is a degenerate pencil}\right\rbrace
\end{equation*}
is a divisor and the class of its closure is given by the following formula
\begin{equation*}
\overline{\mathfrak{D}}_{15,8}= 6\cdot\left(39\cdot\lambda + 17\cdot \psi -b_{irr}\cdot \delta_{irr}-\sum_{i,s\geq 0}b_{i:s}\cdot\sum_{|S|=s}\delta_{i:S}\right),
\end{equation*}
where $b_{irr}, b_{i:s}\geq 7$ for all $i,s\geq 0$.
\end{thm}
Using the pullback of this divisor to $\overline{\mathcal{M}}_{15,9}$, we obtain a new result on the birational geometry of $\overline{\mathcal{M}}_{15,9}$:
\begin{cor}\label{M15,9 is of general type}
The moduli space $\overline{\mathcal{M}}_{15,9}$ is of general type.
\end{cor}

This work is part of my PhD thesis. I am grateful to my advisor Gavril
Farkas, who brought his recent result \cite{FR} to my attention and thus made the class computations in Theorem \ref{class} and \ref{Divisor on M15,8} possible. My special thanks also go to my coadvisor Angela Ortega for fruitful discussions and to my colleague Daniele Agostini, who helped me with the Macaulay computations in the proof of Theorem \ref{Divisor on M15,8}. I have been supported by Berlin Mathematical School during the preparation of this work.

\subsection*{Notation}

In what follows, we will denote by $Q_k(C,\mathbb{P}^r)$ the variety of quadrics of rank at most $k$ containing $C$, if the embedding $C\hookrightarrow\mathbb{P}^r$ is clear from the context. We will write $Q_k(C,L)$ when the $g^r_d$ is complete, i.e $g^r_d=(L,V)$ where $V=H^0(C,L)$. Finally, $Q_k(\mathbb{P}^r)$ will stand for the variety of quadrics having rank at most $k$ in $\mathbb{P}^r$.

\section{Proof of Theorem \ref{main thm}}

It is well known that the Hilbert scheme of curves of genus $g$ and degree $d$ in $\mathbb{P}^r$ has a unique component that parametrizes curves with general moduli, when $\rho(g,r,d)\geq 0$. We let $\mathcal{H}_{d,g,r}$ denote this component. For a projective curve $C\subseteq \mathbb{P}^r$  we define the variety $Q_k(C,\mathbb{P}^r)$ as
\begin{equation*}
Q_k(C,\mathbb{P}^r)\defeq \mathbb{P}\left(I_2\left(C, \mathcal{O}_C(2)\right)\right)\cap Q_k(\mathbb{P}^r)\subseteq |\mathcal{O}_{\mathbb{P}^r}(2)|.
\end{equation*}
Since it is the intersection of two projective varieties in the projective space, all its irreducible components have dimension at least $q(g,r,d,k)$. Therefore Theorem \ref{main thm} will follow once we show that
\begin{equation*}
\dim Q_k(C,g^r_d)\leq q(g,r,d,k)
\end{equation*}
for a general element $[C\subseteq \mathbb{P}^r]\in\mathcal{H}_{d,g,r}$. We show this inductively using nodal curves that are defined as the union of an element $[C'\subseteq \mathbb{P}^r]\in \mathcal{H}_{d-1,g-1,r}$, which satisfies Theorem \ref{main thm}, and a general secant line of $C'$. For a given $r$, the base steps of this inductive argument are the case of rational normal curves when $g-d+r=0$ and the case of canonical curves when $g-d+r=1$, since $\rho(g,r,d)=0$ implies that $g=r+1$ and $d = 2r$. In the following lemmas, we confirm Theorem \ref{main thm} for these two base cases.

\begin{lem}\label{Canonical Curves}
Let $C$ be a general curve of genus $g$. Then $Q_k(C,K_C)$ is of pure dimension $q(g,g-1,2g-2,k)$ for all $k\geq 3$.
\end{lem}
\begin{proof}
Let $C$ be a general curve of genus $g$. For $k=3$, the expected dimension of $Q_k(C,K_C)$ is equal to
\begin{equation}\label{rank3 -1}
q(g,g-1,2g-2,3)=-1.
\end{equation}
Therefore we need to show that there are no rank 3 quadrics containing the canonical model of $C$. A rank 3 quadric in $\mathbb{P}^r$ is ruled by a pencil of $r-2$ planes, where the base locus of the pencil is the singular locus of the quadric. It is not hard to see that every element of this pencil (considered with multiplicity two) is a hyperplane section of the quadric (In fact these hyperplanes are the tangent hyperplanes of the quadric). Therefore if a rank 3 quadric $Q$ contains $C$ then the pencil of $r-2$ planes of $Q$ cuts out a pencil $A$ on $C$ such that
\begin{equation*}
K_C = 2A+F,
\end{equation*}
where $F$ is a divisor supported in $C\cap \sing(Q)$. It follows from the base point free pencil trick (see \cite{ACGH}) that the Petri map
\begin{equation*}
\mu:H^0(A+F)\otimes H^0(A)\to H^0(K_C)
\end{equation*}
has at least one dimensional kernel. However, it is well known that such curves are special in moduli \cite{G}.

For $k=4$, we need to show that 
\begin{equation*}
\dim Q_4(C,K_C)= q(g,g-1,2g-2,4)=g-4.
\end{equation*}
A rank 4 quadric in $\mathbb{P}^r$ has two distinct pencils of $r-2$ planes, both of which sweep out the quadric. Similar to the rank 3 case, the base loci of these pencils are equal to the singular locus of the quadric. Moreover, the union of two $r-2$ planes belonging to two different pencils is a (tangent) hyperplane section of the quadric. Therefore if $Q$ has rank 4 and contains $C$ then the two pencils of $r-2$ planes cut out two pencils $A_1,A_2$ on $C$ such that
\begin{equation*}
K_C = A_1+A_2+F,
\end{equation*}
where, as before, $F$ is a divisor supported on $C\cap \sing(Q)$. Using this correspondence we can estimate the dimension of $Q_4(C,K_C)$ as follows:

To give an element of $Q_4(C,K_C)$, one has to specify a pencil $A$ of degree $a$ and a 2-dimensional space of sections of $H^0(K_C-A)$. Since $C$ is general, we can use the Brill-Noether theorem to count the parameters that these choices depend on:
\begin{equation*}
\dim Q_4(C,K_C) = \dim W^1_a(C)+\dim \textnormal{Gr}(2,g-a+1)= g-4.
\end{equation*}
That finishes the proof for $k=4$.

To deal with the case $k\geq 5$, we let $\widetilde{\mathcal{H}}_{g}$ be the locus of curves in $\mathcal{H}_{g,g-1,2g-2}$, for which we have that $\dim Q_4(C,K_C)=g-4$. We define the incidence variety 
\begin{equation*}
I_4\defeq\left\lbrace(Q,[C\subseteq \mathbb{P}^{g-1}])\mid C\subseteq Q\right\rbrace\subseteq Q_4(\mathbb{P}^{g-1})\times \widetilde{\mathcal{H}}_{g}.
\end{equation*}
Using the projection map $I_4\to \widetilde{\mathcal{H}}_{g}$, we compute that
\begin{equation*}
\dim I_4=3g-3+g^2-1+g-4=g^2+4g-8.
\end{equation*}
Since $\dim Q_4(\mathbb{P}^{g-1})=4g-7$, the dimension of the general fiber of the other projection map
\begin{equation*}
I_4\to Q_4(\mathbb{P}^{g-1})
\end{equation*}
is equal to $g^2 - 1$. Since all rank 4 quadrics are projectively equivalent, we conclude that they all contain $g^2 - 1$ dimensional family of canonical curves.

Next we consider the incidence variety
\begin{equation*}
I\defeq\left\lbrace(Q,[C\subseteq \mathbb{P}^{g-1}])\mid C\subseteq Q\right\rbrace\subseteq |\mathcal{O}_{\mathbb{P}^{g-1}}(2)|\times \widetilde{\mathcal{H}}_{g}.
\end{equation*}
Since canonical curves are projectively normal, $I$ is a projective bundle over $\widetilde{\mathcal{H}}_{g}$. Thus we obtain that $I$ is irreducible of dimension
\begin{equation*}
\dim I = \frac{1}{2} (3g^2 + g -4).
\end{equation*}
The projection map 
\begin{equation*}
I\to |\mathcal{O}_{\mathbb{P}^{g-1}}(2)|
\end{equation*}
is clearly dominant and thus has relative dimension $g^2 -1$ over an open set of $|\mathcal{O}_{\mathbb{P}^{g-1}}(2)|$. By the discussion in the preceeding paragraph, quadrics of rank 4 lie in this open set. Since for every $k\geq 5$ the variety $Q_k(\mathbb{P}^{g-1})$ contains $Q_4(\mathbb{P}^{g-1})$, we can find quadrics of arbitrary rank, which lie in this open set and hence contain a $g^2 -1$ dimensional family of canonical curves. By projective equivalence, this applies to \emph{all} quadrics of rank $k\geq 4$.

Finally we restrict ourselves to quadrics of rank at most $k$ and consider the incidence variety
\begin{equation*}
I_k\defeq\left\lbrace(Q,[C\subseteq \mathbb{P}^{g-1}])\mid C\subseteq Q\right\rbrace\subseteq Q_k(\mathbb{P}^{g-1})\times \widetilde{\mathcal{H}}_{g}.
\end{equation*}
By the above discussion, we know the relative dimension of the map $I_k\to Q_k(\mathbb{P}^{g-1})$. Using this, we compute
\begin{equation*}
\dim I_k = {g+1\choose 2}-{g+1-k\choose 2}+g^2-2.
\end{equation*}
Hence the dimension of the general fiber of $I_k\to \widetilde{\mathcal{H}}_{g}$ is equal to
\begin{equation*}
{g+1\choose 2}-{g+1-k\choose 2}-3g+2.
\end{equation*}
which is the same as $q(g,g-1,2g-2,k)$.
\end{proof}

\begin{lem}\label{RNC's}
For any $k\geq 3$ and any rational normal curve $\Gamma\subseteq\mathbb{P}^r$, the variety $Q_k(\Gamma,\mathbb{P}^r)$ has the expected dimension $q(0,r,r,k)$. 
\end{lem}
\begin{proof}
First we confirm the case $k=3$, that is, we show that for a rational normal curve $\Gamma\subseteq \mathbb{P}^r$, we have that
\begin{equation*}
\dim Q_3(\Gamma,\mathbb{P}^r)=q(0,r,r,3).
\end{equation*}
As explained in the previous lemma, the elements of $Q_3(\Gamma,\mathbb{P}^r)$ are in one to one correspondence with the data $(A,F)$ such that
\begin{equation*}
2A+F = \mathcal{O}_{\Gamma}(1),
\end{equation*}
where $A$ is a pencil and $F$ is a divisor supported on the singular locus of the associated quadric. If we let $x = \deg(F)$ then the parameter count for the pairs $(A,F)$ yields
\begin{equation*}
\dim \textnormal{Gr}\left(2, \frac{r-x}{2}+1\right) + x=r-2.
\end{equation*}
Since $q(0,r,r,3)=r-2$, that finishes the proof of the case $k=3$. The rest of the proof is analogous to the proof of Lemma \ref{Canonical Curves}, we leave these details to the reader.
\end{proof}

In the next proposition we prove the inductive step, which takes care of the cases $g-d+r=0,1$. In the proof we will need the following lemma from \cite{BE1}:
\begin{lem}\label{Ballico lemma}
Let $\rho(g,r,d)\geq 0$ and $0\leq g\leq d-r + \left \lfloor{\frac{d-r-2}{r-2}}\right \rfloor$. If $C\in \mathcal{H}_{d,g,r}$ and $\ell$ is a 2-secant line of $C$ then $C\cup \ell\in \mathcal{H}_{d+1,g+1,r}$.
\end{lem}
\begin{proof}
See Lemma 2.2 in \cite{BE1}.
\end{proof}

\begin{prop}\label{g-d+r=0,1}
Theorem \ref{main thm} holds whenever $\rho(g,r,d)\geq 0$ and $g-d+r=0$ or $1$.
\end{prop}
\begin{proof}
We fix $r\geq 3$ and apply induction on $g$. As we already pointed out, the base step of the induction was confirmed in Lemma \ref{Canonical Curves} and Lemma \ref{RNC's}. For the inductive step, we let $[C\subseteq \mathbb{P}^r]\in\mathcal{H}_{d,g,r}$ such that
\begin{equation*}
\dim Q_k(C,\mathbb{P}^r)=q(g,r,d,k)
\end{equation*}
for all $k\geq 3$. Let $\ell$ be a general 2-secant line of $C$ and consider the nodal curve $X\defeq C\cup \ell$. By Lemma \ref{Ballico lemma}, we have that $[X\subseteq \mathbb{P}^r]\in \mathcal{H}_{d+1,g+1,r}$. Since
\begin{equation*}
q(g+1,r,d+1,k)=q(g,r,d,k)-1,
\end{equation*}
all we need to show is that to contain the secant line $\ell$ imposes a nontrivial condition on the variety $Q_k(C,\mathbb{P}^r)$. This follows from the fact that the secant variety of a non-degenerate curve does not lie in \emph{any} quadric (See Corollary 2.3 in \cite{Ca}).
\end{proof}
The only remaining case is the case of incomplete embeddings (i.e. $g-d+r<0$), which will be treated in the next proposition. We first make a simple observation, which we will use in the proof of the proposition.

\begin{lem}\label{less stupid lemma}
Let $C$ be a smooth curve of genus $g$ and $\ell=(L,V)$ a very ample $g^r_d$ on the curve. Let $W^\vee$ be the kernel of the map $H^0(L)^\vee\to V^\vee$ induced by the inclusion $V\subseteq H^0(L)$. Consider the following commutative diagram
\begin{center}
\begin{tikzpicture}
  \node (C) {$C$};
  \node (d-g) [node distance=2.5cm, right of=C] {$\mathbb{P}(H^0(L)^\vee)$};
  \node (r) [node distance=1.5cm, below of=d-g] {$\mathbb{P}(V^\vee)$};
  \draw[->] (C) to node {$|L|$} (d-g);
  \draw[->] (C) to node [swap] {$\ell$} (r);
  \draw[->] (d-g) to node {$\pi$} (r);
\end{tikzpicture}
\end{center}
where $\pi$ is the projection with center $\mathbb{P}(W^\vee)$. We have that
\begin{equation*}
\dim Q_k(C,\ell)=\dim Q_k(C,L)[\mathbb{P}(W^\vee)], 
\end{equation*}
where 
\begin{equation*}
Q_k(C,L)[\mathbb{P}(W^\vee)]\defeq\left\lbrace Q\in Q_k(C,L)\mid \mathbb{P}(W^\vee)\subseteq \sing (Q)\right\rbrace.
\end{equation*}
\end{lem}
\begin{proof}
There is an evident map $Q_k(C,L)[\mathbb{P}(W^\vee)]\to Q_k(C,\ell)$ defined by projecting quadrics by $\pi$. The inverse of this map is given by assigning $Q$ to the cone over $Q$ with vertex $\mathbb{P}(W^\vee)$.
\end{proof}

\begin{prop}
Theorem \ref{main thm} holds in the range $g-d+r<0$.
\end{prop}
\begin{proof}
We fix integers $g,r,d,k$ such that $g-d+r<0$ and we let $C$ be a general curve of genus $g$. The space of $g^r_d$'s on $C$ is defined as follows
\begin{equation*}
G^r_d(C)\defeq \left\{(L,V)\mid L\in\pic^d(C), V\subseteq H^0(L), \dim V = r+1\right\}.
\end{equation*}
In the range $g-d+r<0$, the variety $G^r_d(C)$ is irreducible and sits over the Picard variety $\pic^d(C)$ as a Grassmann bundle. Therefore a general $g^r_d$ on $C$ is simply a general point on the Grassmannian 
\begin{equation*}
\textnormal{Gr}\left(h^0(L)-r-1,h^0(L)\right)
\end{equation*}
for a general line bundle $L\in\pic^d(C)$.

We fix a general line bundle $L\in\pic^d(C)$. By Proposition \ref{g-d+r=0,1} we have that
\begin{equation*}
\dim Q_k(C,L)=q(g,d-g,d,k).
\end{equation*}
We consider the incidence correspondence
\begin{equation*}
I\defeq\left\lbrace(Q,\Lambda)\mid \Lambda\subseteq \sing Q\right\rbrace\subseteq Q_k(C,L)\times \textnormal{Gr}(d-g-r,d-g+1),
\end{equation*}
and the projection maps
\begin{center}
\begin{tikzcd}[column sep={{{{4em,between origins}}}}]
 & I\arrow{dl}[swap]{\pi_1}\arrow{dr}{\pi_2} & \\
Q_k(C,L) && \textnormal{Gr}(d-g-r,d-g+1)
\end{tikzcd}
\end{center}
The singular locus of a rank $k$ quadric has dimension $d-g-k$ and therefore the relative dimension of $\pi_1$ over the set of quadrics of rank exactly $k$ is equal to the dimension of the Grassmannian $\textnormal{Gr}(d-g-r,d-g-k+1)$. Therefore, we have that
\begin{equation*}
\dim I\geq q(g,d-g,d,k)+(d-g-r)(r+1-k).
\end{equation*}
In fact this is an equality: If $Z$ is a component of $I$ with dimension strictly greater than that number, we must have that
\begin{equation*}
\pi_1(Z)\subseteq Q_{k-1}(C,L).
\end{equation*}
Let $k'$ be the smallest integer such that $\pi_1(Z)\subseteq Q_{k'}(C,L)$. Then a general element of $\pi_1(Z)$ is a quadric of rank $k'$ and by the same dimension count we get that 
\begin{equation*}
\dim Z =q(g,d-g,d,k')+(d-g-r)(r+1-k'),
\end{equation*}
which is strictly smaller than $q(g,d-g,d,k)+(d-g-r)(r+1-k)$. Therefore we conclude that
\begin{equation*}
\dim I= q(g,d-g,d,k)+(d-g-r)(r+1-k).
\end{equation*}
Now there are two cases to consider. First, if $q(g,r,d,k)\geq 0$ then the map $\pi_2$ is surjective, for if $\Lambda\in \textnormal{Gr}(d-g-r,d-g+1)$ then 
\begin{equation*}
\pi_2^{-1}(\Lambda) = Q_k(C,L)[\Lambda],
\end{equation*}
and by Lemma \ref{less stupid lemma} we have that
\begin{equation*}
\dim Q_k(C,L)[\Lambda]=\dim Q_k(C,\mathbb{P}^r)\geq q(g,r,d,k)\geq 0.
\end{equation*}
Therefore, the dimension of the fiber of $\pi_2$ at a general point is equal to
\begin{equation*}
q(g,d-g,d,k)+(d-g-r)(r+1-k)-\dim \textnormal{Gr}(d-g-r,d-g+1)=q(g,r,d,k).
\end{equation*}
Hence we conclude that $\dim Q_k(C,\mathbb{P}^r)=q(g,r,d,k)$. 

On the other hand, if $q(g,r,d,k)<0$ then $\pi_2$ is not surjective and thus for a general element $\Lambda\in \textnormal{Gr}(d-g-r,d-g+1)$, the variety $Q_k(C,L)[\Lambda]$ and hence $Q_k(C,\mathbb{P}^r)$ is empty.
\end{proof}
%%%%%%%%%%%%%%%%%%%%%%%%%%%%%%%%%%%%%%%%%%%%%%%%%%%%%%%%%%%%%%%%%%%%%%%%%
%%%%%%%%%%%%%%%%%%%%%%%%%%%%%%%%%%%%%%%%%%%%%%%%%%%%%%%%%%%%%%%%%%%%%%%%%%
%%%%%%%%%%%%%%%%%%%%%%%%%%%%%%%%%%%%%%%%%%%%%%%%%%%%%%%%%%%%%%%%%%%%%%%%%%
%%%%%%%%%%%%%%%%%%%%%%%%%%%%%%%%%%%%%%%%%%%%%%%%%%%%%%%%%%%%%%%%%%%%%%%%%%
%%%%%%%%%%%%%%%%%%%%%%%%%%%%%%%%%%%%%%%%%%%%%%%%%%%%%%%%%%%%%%%%%%%%%%%%%%
%%%%%%%%%%%%%%%%%%%%%%%%%%%%%%%%%%%%%%%%%%%%%%%%%%%%%%%%%%%%%%%%%%%%%%%%%%
\section{Computation of the class $\overline{\mathfrak{Quad}}^k_{g,n}$}

In the computation of the class $\overline{\mathfrak{Quad}}^k_{g,n}$, we make essential use of a recent result by Farkas and Rim\'{a}nyi:
\begin{thm}[\cite{FR}]\label{Gabi&Rimanyi}
Let $X$ be an algebraic variety and suppose that there is a morphism of vector bundles
\begin{equation*}
\phi : \sym^2 \mathcal{E} \to \mathcal{F}
\end{equation*}
on $X$, where $\mathcal{E}$ and $\mathcal{F}$ are vector bundles of ranks $e$ and $f$, respectively. For $k \leq e$, we define the subvariety
\begin{equation*}
\Sigma^k_{e,f}(\phi)\defeq \left\lbrace x\in X\mid \exists \textnormal { }0\neq q\in \ker(\phi(x))\quad \textnormal{s.t.}\quad\textnormal{rk}(q)\leq k\right\rbrace.
\end{equation*}
If $f = {e+1\choose 2} - {e-k+1\choose 2}$ then $\Sigma^k_{e,f}(\phi)$ is expected to be of codimension one in $X$ and its virtual class is given by the formula
\begin{equation*}
[\Sigma^k_{e,f}(\phi)] = A^k_e\cdot \left(c_1(\mathcal{F})-\frac{2f}{e}\cdot c_1(\mathcal{E})\right),
\end{equation*}
where
\begin{equation*}
A^k_e\defeq \prod_{t=0}^{e-k-1}\frac{{e+i\choose e-k-t}}{{2t+1\choose t}}.
\end{equation*}
\end{thm}
\noindent To define the locus $\mathfrak{Quad}^k_{g,n}$, we fix integers $g,n,k$ such that
\begin{equation*}
q(g,g-n-1,2g-2-n,k)=-1.
\end{equation*}
We let
\begin{equation*}
\pi :\overline{\mathcal{M}}_{g,n+1}\to\overline{\mathcal{M}}_{g,n}
\end{equation*}
be the map that forgets the last marked point and $\mathscr{L}$ be the cotangent line bundle on $\overline{\mathcal{M}}_{g,n+1}$. We define the sheaves
\begin{equation*}
\mathcal{E}\defeq \pi_*\mathscr{L}\left(-\sum_{j=1}^{n}\delta_{0:\{j,n+1\}}\right),
\end{equation*}
and
\begin{equation*}
\mathcal{F}\defeq \pi_*\mathscr{L}^{\otimes 2}\left(-2\cdot\sum_{j=1}^{n}\delta_{0:\{j,n+1\}}\right),
\end{equation*}
and consider the natural multiplication map
\begin{equation}\label{vbmap}
\phi: \sym^2\mathcal{E}\to \mathcal{F}.
\end{equation}
We define the locus $\mathfrak{Quad}^k_{g,n}$ as 
\begin{equation*}
\mathfrak{Quad}^k_{g,n}\defeq \Sigma^k_{e,f}(\phi)\cap\mathcal{M}_{g,n},
\end{equation*}
where $e= g-n$ and $f = 3g-3-2n$. Using Theorem \ref{Gabi&Rimanyi} and Grothendieck-Riemann-Roch formula we obtain the following result:
\begin{thm}\label{Lambda,Psi,Deltairr coefficients}
The coefficients of the class $\overline{\mathfrak{Quad}}^k_{g,n}$ satisfy the following relations:
\begin{equation*}
\alpha^k_{g,n}=\prod_{t=0}^{g-n-k-1}\frac{{g-n+t\choose g-n-k-t}}{{2t+1\choose t}},\quad a = \frac{7 g-9 n+6}{g-n},\quad c = \frac{g+n-6}{g-n},\quad b_{irr} = 1,
\end{equation*}
and $b_{i:s}\geq 1$ for $0\leq i\leq g$ and $0\leq s\leq n$.
\end{thm}
\begin{proof}
The proof follows the same line of arguments as in the proof of Theorem 2.1 in \cite{Ka}. First note that the evaluation map
\begin{equation*}
\pi_*\mathscr{L}\xrightarrow{ev} \pi_*\left(\mathscr{L}|_{\sum_{j=1}^{n}\delta_{0:\{j,n+1\}}}\right)
\end{equation*}
fails to be surjective over the boundary divisor $\Delta_{i:S}$ when $i<s$ or $g-i<n-s$, where $s=|S|$. Therefore we break our analysis into two parts. First, we let $\widetilde{\mathcal{M}}_{g,n}$ be the partial compactification defined as the union of $\mathcal{M}_{g,n}$ with the boundary divisors $\Delta_{i:S}$, such that $s\leq i$ and $n-s\leq g-i$. We will deal with the case where $i<s$ or $g-i<n-s$ later.

The sheaf $\mathcal{E}$ sits in the following sequence of sheaves over $\widetilde{\mathcal{M}}_{g,n}$:
\begin{equation*}
0 \rightarrow \mathcal{E}\rightarrow \pi_*\mathscr{L}\xrightarrow{ev} \pi_*\left(\mathscr{L}|_{\sum_{j=1}^{n}\delta_{0:\{j,n+1\}}}\right)\rightarrow 0.
\end{equation*}
Since we are in the range $s\leq i$ and $n-s\leq g-i$, it easily follows that the evaluation map $ev$ is surjective in codimension 2. Therefore we obtain that
\begin{equation*}
c_1(\mathcal{E})=\lambda-\sum_{j=1}^{n}\psi_j.
\end{equation*}
Next we use Grothendieck-Riemann-Roch formula to compute $c_1(\mathcal{F})$. We let 
\begin{equation*}
\mathscr{M} = \mathscr{L}^{\otimes 2}\left(-2\cdot\sum_{j=1}^{n}\delta_{0:\{j,n+1\}}\right)
\end{equation*}
and write
\begin{equation*}
\textnormal{ch}\left(\pi_! \mathscr{M} \right) = \pi_*\left(\textnormal{ch}(\mathscr{M})\cdot \textnormal{Td}(\pi) \right).
\end{equation*}
Collecting degree 1 terms on both sides, we obtain that
\begin{equation*}
c_1(\mathcal{F}) = \pi_*\left[\left(1 + c_1(\mathscr{M}) + \frac{c_1(\mathscr{M})^2}{2}\right)\cdot \left(1 - \frac{c_1(\mathscr{L})}{2} + \frac{c_1(\mathscr{L})^2 + \Sigma}{12}\right)\right],
\end{equation*}
where $\Sigma$ denotes the locus of pointed singular curves in $\overline{\mathcal{M}}_{g,n+1}$ where the point with label $n+1$ hits the singular locus of the curve. Using basic intersection theory and well known formulas relating push forwards of codimension 2 classes in $\overline{\mathcal{M}}_{g,n+1}$ with divisor classes in $\overline{\mathcal{M}}_{g,n}$ (see Section 17 in \cite{ACG2}) we obtain that 
\begin{equation*}
c_1(\mathcal{F})=13\cdot\lambda -5\cdot \sum_{j=1}^{n}\psi_j-\delta,
\end{equation*}
where $\delta$ denotes the class of the whole boundary. Using Theorem \ref{Gabi&Rimanyi} we obtain the following formula for the class of $\Sigma^k_{e,f}(\phi)$:
\begin{equation}\label{Outcome of GRR+Porteous}
[\Sigma^k_{e,f}(\phi)] = A^k_{g-n}\cdot\left(\frac{7 g-9 n+6}{g-n}\cdot\lambda + \frac{g+n-6}{g-n}\cdot \sum_{j=1}^{n}\psi_j - \delta\right).
\end{equation}
This way we obtained the coefficients $a$ and $c$. To conclude that $b_{irr} = 1$, one also needs to show that $\Delta_{irr}\not\subset \Sigma^k_{e,f}(\phi)$. The arguments in Lemma \ref{Canonical Curves} and Proposition \ref{g-d+r=0,1} can be repeated verbatim for the canonical image of an irreducible nodal curve to show that Theorem \ref{main thm} holds true also for general elements of $\Delta_{irr}$. We skip these details.

From the expression (\ref{Outcome of GRR+Porteous}), we can read off the bound $b_{i:s}\geq 1$ whenever we have that $s\leq i$ and $n-s\leq g-i$. To deal with the boundary coefficients $b_{i:s}$ when $i<s$ or $g-i<n-s$, we introduce the twist
\begin{equation*}
\mathscr{L}'\defeq \mathscr{L}\left(\sum_{\substack{0\leq i\leq g\\ i< s\\|S| = s}}(i-s-1)\cdot\delta_{i:S\cup\{n+1\}}\right),
\end{equation*}
and define the sheaves
\begin{equation*}
\mathcal{E'}\defeq \pi_*\mathscr{L'}\left(-\sum_{j=1}^{n}\delta_{0:\{j,n+1\}}\right),
\end{equation*}
and
\begin{equation*}
\mathcal{F'}\defeq \pi_*\mathscr{L'}^{\otimes 2}\left(-2\cdot\sum_{j=1}^{n}\delta_{0:\{j,n+1\}}\right).
\end{equation*}
It can easily be checked that the ranks of $\mathcal{E'}$ and $\mathcal{F'}$ stay constant away from loci of codimension $2$ and thus we have a morphism of vector bundles in codimension $2$
\begin{equation}\label{vbmap2}
\phi': \sym^2\mathcal{E'}\to \mathcal{F'}.
\end{equation}
extending $\phi$. Between the classes $[\Sigma^k_{e,f}(\phi')]$ and $\overline{\mathfrak{Quad}}^k_{g,n}$, we have the relation
\begin{equation}\label{d_i:s relation}
[\Sigma^k_{e,f}(\phi')]=\overline{\mathfrak{Quad}}^k_{g,n}+\sum d_{i:s}\cdot\delta_{i:S},
\end{equation}
where $d_{i:s}\geq 0$. Letting $\tilde{b}_{i:s}$ denote the coefficient of $\delta_{i:S}$ in the class $[\Sigma^k_{e,f}(\phi')]$, we obtain the equality $\tilde{b}_{i:s} = b_{i:s}-d_{i:s}$.

We computed the intersection of the Chern classes of $\mathcal{E'}$ and $\mathcal{F'}$ with simple test curves in Lemma 2.2 in \cite{Ka}. Using that and Theorem \ref{Gabi&Rimanyi}, we compute
\begin{equation}\label{tilde b_i:s}
\tilde{b}_{i:s} = \frac{1}{g-n}\left(-i^2 (g-2 n+3)+i (2 g-2 s n+6 s-3 n+3)+s ((g-3) s+n-3)\right),
\end{equation}
for all $i,s$ such that $i<s$. It is easy to see that this quantity is always $\geq 1$.
\end{proof}
We believe that the coefficients $d_{0:s}$ in equation (\ref{d_i:s relation}) are equal to zero and therefore $b_{0:s} = \tilde{b}_{0:s}$, but we could prove this only for $k=4$:
\begin{thm}
For $k=4$ and $S\subseteq\{1,\dots ,n\}$ the general element of $\Delta_{0:S}$ does not lie in $\Sigma^k_{e,f}(\phi')$ and we have that
\begin{equation*}
b_{0:s}= \frac{s (gs - 3s + n - 3)}{g-n}\cdot
\end{equation*}
\end{thm}
\begin{proof}
Let $[C,p_{i_1}, \dots , p_{i_k}]$ and $[R, p_{j_1}, \dots , p_{j_\ell}]$ be general pointed curves in $\mathcal{M}_{g,S^c}$ and in $\mathcal{M}_{0,S}$, respectively. We choose general points $p_0\in C$ and $q_0\in R$ and identify these to obtain the pointed curve $[X,p_1,\dots p_n]\in\Delta_{0:S}$. The fiber of the vector bundle map (\ref{vbmap2}) over this moduli point is equal to
\begin{equation*}
\sym^2 H^0\left(K_C - \sum_{i\in S^c} p_i - s\cdot p_0\right)\to H^0\left(K_C^{\otimes 2}- 2\cdot\sum_{i\in S^c} p_i - 2s\cdot p_0\right).
\end{equation*}
Therefore, we need to show that there are no rank 4 quadrics containing the image of $C$ under the map given by the linear series $|K_C - \sum_{i\in S^c} p_i - s\cdot p_0|$. It is clearly sufficient to specialize to the case where $p_i = p_0$ for all $i\in S^c$ and check this claim for the linear series $|K_C - n\cdot p_0|$.

First note that for $k=4$ the numerical condition $q(g,g-n-1,2g-2-n,k)=-1$ implies that $g=2n+3$. Let $A$ be a $g^1_a$ on the curve $C$ such that the pencil pair $(A, K_C-A-n\cdot p_0)$ corresponds to a rank 4 quadric containing $C$. By Brill-Noether theory, we have that $a\geq n+3$. Since $|K_C - A|$ is a $g^{g-a}_{2g-2-a}$, the condition that
\begin{equation*}
h^0(K_C-A-n\cdot p_0)\geq 2
\end{equation*}
imposes non trivial conditions on the ramification type of the linear series $|K_C-A|$ at the point $p_0$. Precisely, we have the following inequality for the ramification sequence of $|K_C-A|$:
\begin{equation}\label{Ramification Type}
\alpha^{K_C-A}(p_0)\geq (0,\dots ,0, a-n-2,a-n-2).
\end{equation}
We know from Brill-Noether theory that a $g^r_d$ subject to a ramification condition
\begin{equation*}
\alpha^{g^r_d}(p)\geq (\alpha_0,\alpha_1,\dots , \alpha_r)
\end{equation*}
at a general point $p\in C$ has non-negative adjusted Brill-Noether number, which is defined as
\begin{equation*}
\rho(g,r,d;\alpha) = g - (r+1)(g-d+r) - \sum_{i=0}^r \alpha_i
\end{equation*}
(See Proposition 1.2 in \cite{EH} and the discussion preceding it). Using the inequality (\ref{Ramification Type}), we estimate the adjusted Brill-Noether number of $|K_C - A|$ at $p_0$ as
\begin{equation*}
\rho(g,g-a,2g-2-a;\alpha^{K_C-A}(p_0))\leq -1.
\end{equation*}
Therefore for a general pointed curve $[C,p_0]\in\mathcal{M}_{g,1}$ such a linear series does not exist, which in turn implies that $C$ is not contained in a rank 4 quadric when mapped to the projective space via the linear series $|K_C - n\cdot p_0|$.
\end{proof}
In the proof of Theorem \ref{Divisor on M15,8} we use the following result from \cite{FR}:
\begin{thm}[\cite{FR}]\label{Degenerate pencil class}
Let $X$ be an algebraic variety and
\begin{equation*}
\phi:\sym^2(\mathcal{E})\rightarrow \mathcal{F}
\end{equation*}
be a morphism of vector bundles over $X$ where $\textnormal{rk}(\mathcal{E})=e$ and $\textnormal{rk}(\mathcal{F})={e+1\choose 2}-2$. The class of the virtual divisor 
\begin{equation*}
\mathfrak{Dp}(\phi):=\left\lbrace x\in X\mid \ker(\phi(x)) \textnormal{ is a degenerate pencil}\right\rbrace
\end{equation*}
is given by the formula
\begin{equation*}
[\mathfrak{Dp}(\phi)]=(e-1)\cdot \left(e\cdot c_1(\mathcal{F})-(e^2+e-4)\cdot c_1(\mathcal{E})\right)\in H^2(X,\mathbb Q).
\end{equation*}
\end{thm}
\begin{proof}[Proof of Theorem \ref{Divisor on M15,8}]
We first show that 
\begin{equation*}
\overline{\mathfrak{D}}_{15,8}\neq \overline{\mathcal{M}}_{15,8}.
\end{equation*}
To do this it suffices to exhibit an element $[C,p_1,\dots , p_8]\in\mathcal{M}_{15,8}$ such that the image of $C$ under the map induced by $|K_C-p_1-\dots - p_8|$ lies in a non-degenerate pencil of quadrics. In other words, it is sufficient to show that there exists a smooth curve $C\subseteq\mathbb{P}^6$ of genus $15$ and degree $20$ such that the pencil $I_2(C, \mathcal{O}_C(2))$ is non-degenerate. To this end, we pick 15 general points on $\mathbb{P}^2$ and consider the blown up surface $X\defeq \textnormal{Bl}_{15}(\mathbb{P}^2)$. Using Macaulay, we show that the linear system
\begin{equation*}
H = 7h - 2(E_1+\dots + E_7) -E_8-\dots - E_{15}
\end{equation*}
embeds $X$ to $\mathbb{P}^6$, where $h$ is the class of a line in $\mathbb{P}^2$ and $E_i$ is the exceptional divisor corresponding to the  $i^{th}$ point. We also check that $\dim I_2(X, \mathcal{O}_X(2))=2$ and that the pencil is non-degenerate. Next, we let $C$ to be a general element of the linear system
\begin{equation*}
\left| 10h-3(E_1+E_2+E_3)-2(E_4+\dots +E_{15})\right|.
\end{equation*}
Again using Macaulay we check that this linear system is base point free, hence by Bertini's theorem we obtain that $C$ is smooth. It is easy to see that the genus of $C$ is $15$ and $C.H=20$. We also check that
\begin{equation*}
H^0(\mathcal{O}_X(2H-C))=0, 
\end{equation*}
which implies that the restriction map
\begin{equation*}
\rho:H^0(\mathcal{O}_X(2))\to H^0(\mathcal{O}_C(2))
\end{equation*}
is injective. Therefore the map
\begin{equation*}
H^0(\mathcal{O}_{\mathbb{P}^6}(2))\to H^0(\mathcal{O}_C(2))
\end{equation*}
factors through $\rho$ and it follows that
\begin{equation*}
I_2(X,\mathcal{O}_X(2))=I_2(C, \mathcal{O}_C(2)).
\end{equation*}
To compute the class of $\overline{\mathfrak{D}}_{15,8}$, we let
\begin{equation*}
\mathcal{E}\defeq \pi_*\mathscr{L}\left(-\sum_{j=1}^{8}\delta_{0:\{j,9\}}\right),
\end{equation*}
and
\begin{equation*}
\mathcal{F}\defeq \pi_*\mathscr{L}^{\otimes 2}\left(-2\cdot\sum_{j=1}^{8}\delta_{0:\{j,9\}}\right),
\end{equation*}
and consider the morphism of vector bundles
\begin{equation}
\phi: \sym^2\mathcal{E}\to \mathcal{F}
\end{equation}
over the partial compactification $\widetilde{\mathcal{M}}_{g,n}$ (Here we use the notation in the proof of Theorem \ref{Lambda,Psi,Deltairr coefficients}). Using Theorem \ref{Degenerate pencil class}, we compute that
\begin{equation*}
[\mathfrak{Dp}(\phi)] = 6\cdot\left(39\cdot\lambda +17\cdot\sum_{j=1}^8 \psi_j-7\cdot\delta\right).
\end{equation*}
It follows that the boundary coefficients $b_{irr}$ and $b_{i:s}$ of $\overline{\mathfrak{D}}_{15,8}$ are at least $7$ when $s\leq i$ and $ n-s\leq g-i$. To obtain a lower bound for $b_{i:s}$ in the remaining cases we use the extension (\ref{vbmap2}) of the vector bundle map $\phi$. Analogously to the proof of Theorem \ref{Lambda,Psi,Deltairr coefficients}, we write
\begin{equation}
[\mathfrak{Dp}(\phi')]=6\cdot\left(39\cdot\lambda +17\cdot\sum_{j=1}^8 \psi_j-7\cdot \delta_{irr}-\sum_{i,s\geq 0}\tilde{b}_{i:s}\cdot\sum_{|S|=s}\delta_{i:S}\right),
\end{equation}
and compute that
\begin{equation*}
\tilde{b}_{i:s} = -2 i^2+i (9-10 s)+s (12 s+5)
\end{equation*}
for all $i,s$ such that $i<s$. Clearly, $b_{i:s}\geq \tilde{b}_{i:s}\geq 7$.
\end{proof}
We conclude the paper with the proof of Corollary \ref{M15,9 is of general type}.
\begin{proof}[Proof of Corollary \ref{M15,9 is of general type}]
We first recall that the canonical class of $\overline{\mathcal{M}}_{g,n}$ is well known to be
\begin{equation*}
K_{\overline{\mathcal{M}}_{g,n}}=13\cdot\lambda-2\cdot\delta_{irr}+ \sum_{j=1}^n\psi_j-2\cdot \sum_{\substack{S\in P\\|S|\geq 2}}\delta_{0:S}-3\cdot \sum_{S\in P}\delta_{1:S}-2\cdot\sum_{i=2}^{\left \lfloor{g/2}\right \rfloor}\sum_{S\in P}\delta_{i:S},
\end{equation*}
where $P$ denotes the power set of $\{1,\dots,n\}$. We consider the map
\begin{equation*}
\pi_j:\overline{\mathcal{M}}_{15,9}\to \overline{\mathcal{M}}_{15,8}
\end{equation*}
that forgets the point labeled by $j$. Pulling back the divisor $\overline{\mathfrak{D}}_{15,8}$ via $\pi_j$ for every $j\in\{1,\dots ,9\}$ and taking their average, we obtain the effective class
\begin{equation*}
\mathcal{Z}_{15,9} = 351\cdot\lambda +136\cdot \sum_{j=1}^9\psi_j-b_{irr}\cdot \delta_{irr}-\sum_{i,s\geq 0}b_{i:s}\cdot\sum_{|S|=s}\delta_{i:S},
\end{equation*}
where $b_{irr},b_{i:s}\geq 63$ for all $i,s\geq 0$.

On $\overline{\mathcal{M}}_{15,9}$ we also have the pullback of the Brill-Noether divisor from $\overline{\mathcal{M}}_{15}$, whose class is given by the following formula:
\begin{equation*}
BN_{15} = 54\cdot\lambda - 8\cdot\delta_{irr}-\dots
\end{equation*}
Using these classes we can express $K_{\overline{\mathcal{M}}_{15,9}}$ as follows:
\begin{equation*}
K_{\overline{\mathcal{M}}_{15,9}} = \frac{25}{297}\cdot\sum_{j=1}^9\psi_j + \frac{2}{297}\cdot \mathcal{Z}_{15,9} + \frac{13}{66}\cdot BN_{15} + E,
\end{equation*}
where $E$ is an effective divisor supported on the boundary of $\overline{\mathcal{M}}_{15,9}$. Since the class $\sum_{j=1}^9\psi_j$ is \emph{big}, the result follows.
\end{proof}
%%%%%%%%%%%%%%%%%%%%%%%%%%%%%%%%%%%%%%%%%%%%%%%%%%%%%%%%%%%%%%%%%%%%%%%%%
%%%%%%%%%%%%%%%%%%%%%%%%%%%%%%%%%%%%%%%%%%%%%%%%%%%%%%%%%%%%%%%%%%%%%%%%%%
%%%%%%%%%%%%%%%%%%%%%%%%%%%%%%%%%%%%%%%%%%%%%%%%%%%%%%%%%%%%%%%%%%%%%%%%%%
%%%%%%%%%%%%%%%%%%%%%%%%%%%%%%%%%%%%%%%%%%%%%%%%%%%%%%%%%%%%%%%%%%%%%%%%%%
%%%%%%%%%%%%%%%%%%%%%%%%%%%%%%%%%%%%%%%%%%%%%%%%%%%%%%%%%%%%%%%%%%%%%%%%%%
%%%%%%%%%%%%%%%%%%%%%%%%%%%%%%%%%%%%%%%%%%%%%%%%%%%%%%%%%%%%%%%%%%%%%%%%%%

\bibliographystyle{amsalpha}

\begin{thebibliography}{999999}

\bibitem [ACGH]{ACGH} E. Arbarello, M. Cornalba, P. Griffiths, J. Harris: \emph{Geometry of algebraic curves, Volume I}, Grundlehren der mathematischen Wissenschaften \textbf{267}, Springer-Verlag (1985).

\bibitem [ACG2]{ACG2} E. Arbarello, M. Cornalba, P. Griffiths: \emph{Geometry of algebraic curves, Volume II}, Grundlehren der mathematischen Wissenschaften \textbf{268}, Springer-Verlag (2011).

\bibitem [AM]{AM} A. Andreotti, A.L. Mayer: \emph{On period relations for abelian integrals on algebraic curves}, Annali della Scuola Normale Superiore di Pisa \textbf{21} (1967), 189-238. 

\bibitem [BE1]{BE1} E. Ballico, P. Ellia: \emph{On the existence of curves with maximal rank in $\mathbb{P}^n$}, J. Reine Angew. Math. \textbf{397} (1989), 1-22.

\bibitem[BF]{BF} E. Ballico, C. Fontanari: \emph{Normally generated line bundles on general curves II}, J. Pure Appl. Algebra \textbf{214} (2010), 1450-1455.

\bibitem[Ca]{Ca} M. L. Catalano-Johnson: \emph{The homogeneous ideals of higher secant varieties}, Journal of Pure and Applied Algebra \textbf{158} (2001), 123-129.

\bibitem [EH]{EH} D. Eisenbud, J. Harris: \emph{The Kodaira dimension of the moduli space of curves of genus $\geq 23$}, Inventiones Math. \textbf{90} (1987), 359-387.

\bibitem [FR]{FR} G. Farkas, R. Rim\'{a}nyi: \emph{Quadric rank loci on moduli of curves and K3 surfaces}, arXiv:1707.00756 (2017).

\bibitem [G]{G} D. Gieseker: \emph{Stable curves and special divisors}, Invent. Math. \textbf{66} (1982), 251-275.

\bibitem [Ha]{Ha} J. Harris: \emph{Curves in projective space}, Presses de l'Universit\'{e} de Montr\'{e}al 1982.

\bibitem [HT]{HT} J. Harris, L. W. Tu: \emph{On symmetric and skew-symmetric determinantal varieties}, Topology \textbf{23} (1984), 71-84.

\bibitem[JP]{JP} D. Jensen, S. Payne: \emph{Tropical independence II: The maximal rank conjecture for quadrics}, Algebra Number Theory 10 (2016), No. 8, 1601-1640.

\bibitem[Ka]{Ka} \.I. Kad{\i}k\"oyl\"u: \emph{Maximal rank divisors on $\overline{\mathcal{M}}_{g,n}$}, arXiv:1705.04250 (2017).

\bibitem [Za]{Za} A. G. Zamora: \emph{On the variety of quadrics of rank four containing a projective curve}, Bollettino dell'Unione Matematica Italiana  Serie 8 Vol. \textbf{2-B} (1999), 453-462.

\end{thebibliography}

\end{document}